\documentclass[12pt]{article}

\usepackage{algorithm,algpseudocode}
\usepackage{amscd,amsfonts,amsopn,amssymb,amstext}
\usepackage{appendix}
\usepackage{booktabs}
\usepackage{bbm,bm}
\usepackage[centertags]{amsmath}
\usepackage{color}
\usepackage{diagbox}
\usepackage{fullpage}
\usepackage{graphicx,graphics,psfrag}
\usepackage{indentfirst}
\usepackage{latexsym,enumerate}
\usepackage{lscape}
\usepackage{multirow}
\usepackage{natbib}
\usepackage{rotating}
\usepackage{setspace}
\usepackage{srcltx}
\usepackage{subfigure}
\usepackage[T1]{fontenc}
\usepackage{threeparttable}
\usepackage{times}
\usepackage{url}
\usepackage{verbatim}

\makeatletter
\renewcommand\@biblabel[1]{#1.}
\makeatother

\newtheorem{lemma}{\vspace{-0.2cm}\newline{Lemma}}

\newtheorem{theorem}{Theorem}

\newtheorem{corollary}{Corollary}
\newenvironment{proof}{{\bf Proof:\ \ }}{\qed}
\newcommand{\qed}{\rule{0.5em}{1.5ex}}

\title{\Large \bf{
New Representations of Catalan's Constant, Apery's Constant and the Euler Numbers Obtained from the Half Hyperbolic Secant Distribution }}

\author{\normalsize
{\bf Emilio G\'omez-D\'eniz$\,^a$, Jos\'e Mar\'ia Sarabia$\,^b$}\\
{\small $\,^a$Department of Quantitative Methods in Economics and TIDES Institute}\\[-0.2cm]
{\small University of Las Palmas de Gran Canaria, Spain}\\[-0.2cm]
{\small $\,^b$Department of Economics, University of Cantabria, Santander, Spain}\\[-0.2cm]
}

\date{}

\begin{document}

\maketitle

\begin{abstract}
\noindent New expressions and bounds for Catalan's and Apery's constants, derived from the half hyperbolic secant distribution, are presented. These constants are obtained by using expressions for the Lorenz curve, the Gini and Theil indices, convolutions and a mixture of distributions, among other approaches. The new expressions are presented both in terms of integral (simple and double) representation and also as an interesting series representation. Some of these features are well known, while others are new. In addition, some integral representations of Euler's numbers are obtained.

\vspace{-0.5cm}
\paragraph{Keywords:} {Apery's Constant; Catalan's Constant; Euler Numbers; Gini Index; Half-hyperbolic secant distribution, Lorenz Curve}

\end{abstract}

\section{Introduction}
Among the important special functions in analytic number theory, the Dirichlet beta function given by
\begin{eqnarray}
\beta(\gamma)=\sum_{k=0}^{\infty}\frac{(-1)^k}{(1+2k)^{\gamma}}=\frac{1}{\Gamma(\gamma)}\int_0^{\infty}\frac{x^{\gamma-1} e^{-x}}{1+e^{-2x}}\,dx,\quad \gamma>0\label{beta}
\end{eqnarray}
appears to be a generalisation of Catalan's constant, defined as $\beta(2)$. Hence, the constant is given by
\begin{eqnarray}
{\bf G}=\beta(2)=\sum_{k=0}^{\infty}\frac{(-1)^k}{(1+2k)^2},\label{catalan}
\end{eqnarray}
with an approximate value of  ${\bf G}\approx 0.915965$. As is well known, Catalan's constant can be represented through many rational series and integral representations. Moreover, this constant appears in numerous scientific disciplines, including topology, number theory and statistical mechanics. The most basic integral representation of ${\bf G}$ is given by \citetext{\citealp{jamesonandlord_2017}}
\begin{eqnarray}
{\bf G}=\int_0^1\frac{\tan^{-1} x}{x}\,dx,\label{ir}
\end{eqnarray}
where $\tan^{-1}$ represents the inverse circular tangent function and which is obtained easily from the alternating series representation of the $\tan^{-1}$ function.
\begin{eqnarray}
\tan^{-1}(z)=\sum_{k=0}^{\infty}(-1)^k\frac{x^{1+2 k}}{1+2k}.\label{srt}
\end{eqnarray}

However, a simple integral representation can also be derived from \eqref{beta} assuming $\gamma=2$. The integral \eqref{ir} was extensively studied by Ramanujan. Some other simple integral representations of Catalan's constant can be found in \cite{bradley_2001}, \cite{ferrettietal_2020} and \cite{jamesonandlord_2017}.
The Catalan constant was studied by Ramanujan, by considering several acceleration formulas \citetext{see \citealp{berndt_1985} and \citealp{ramanujan_1915}}. This problem has also been studied by \cite{bradley_1999}. On the other hand, Ramanujan also provided acceleration formulas for Ap\'ery's constant \citetext{see \citealp{berndt_1989}}.

Euler's numbers, also known as the secant numbers or zig numbers, are defined as the coefficients of the power series \citetext{\citealp{becerraanddeeba_2008}}
\begin{eqnarray*}
\sec z = \sum_{k=0}^{\infty}(-1)^k\frac{E_{2k}z^{2k}}{(2k)!},
\end{eqnarray*}
where $\sec(\cdot)$ is the circular secant function \citetext{see for instance, \citealp[p.85]{abramowitzandstegun_1964} and \citealp{borweinetal_1989}}. In particular, $E_0=1$, $E_2= -1$, $E_4=5$, $E_6=-61,\dots$, while $E_k=0$ when $k$ is odd.

Because it will be used later, we recall that for a random variable $X>0$ with cdf $F_X(x)=\Pr(X\leq x)$ its mean (if it exists) can be computed as
\begin{eqnarray}
E_X(X) &=& \int_0^{\infty}(1-F_X(t))\,dt, \label{m1}\\
E_X(X) &=& \int_0^{1}F_X^{-1}(t)\,dt, \label{m2}
\end{eqnarray}
and all the raw moments (if they exist) can be computed as
\begin{eqnarray}
E_X(X^r)=\int_0^{1}[F_X^{-1}(t)]^r\,dt,\quad r=1,2,\dots\label{rm}
\end{eqnarray}

The aim of this paper is to obtain integral (simple and double) representations of Catalan's and Appery's constants. Some are known in the literature but most are new. Fundamentally, these expressions are obtained by using, on the one hand, properties of the half hyperbolic secant distribution and, on the other, inequality measures obtained for this distribution that are commonly used in Economics, especially in the context of income distribution. Tools based on convolution and a mixture of distributions are also used to derive representations of Catalan's constant.

Some integral representations of Euler's numbers are also obtained. In addition, and collaterally, some rational series that may be useful in the literature are obtained. One of these enables us to obtain a new representation of $\xi(2)$ that generates its numerical value $\pi^2/6$, a longstanding question known as the Basel problem.

The rest of the paper is structured as follows. In Section \ref{s2}, we introduce the half hyperbolic secant distribution, showing that some of its features are associated with Catalan's constant and that the even moments of this distribution coincide with Euler's numbers. The convolution of two half hyperbolic secant distributions provides other representations of Catalan's constant and a new representation of $\xi(2)$. Section \ref{s3} focuses on measures of inequality and their relation with Catalan's and Appery's constants. A mixture of distributions in which the half hyperbolic secant distribution acts as the mixing distribution is studied in Section \ref{s4}. This tool provides a new representation of Catalan's constant regarding series and integrals. By restricting the support of the half hyperbolic distribution to the interval $(0,1)$, we produce additional representations of the constant.

\section{Half hyperbolic secant distribution}\label{s2}
We begin by considering the half hyperbolic secant distribution, which is the half (positive) version of the hyperbolic secant distribution \citetext{\citealp{ding_2014} and \citealp{holst_2013}}, a distribution associated with the Cauchy distribution \citetext{\citealp{manoukianandnadeau_1988}}.
The half hyperbolic secant distribution has a probability density function (pdf) and cumulative distribution function (cdf) given by
\begin{eqnarray}
f_X(x) &=& \mbox{sech}\left(\frac{\pi x}{2}\right)=\frac{2}{e^{\pi x/2}+e^{-\pi x/2}},\quad x\geq 0,\label{pdfhshd}\\
F_X(x) &=& \Pr(X\leq x)=\frac{2}{\pi}\tan^{-1}\left(\sinh\left(\frac{\pi x}{2}\right)\right)=\frac{4}{\pi}\tan^{-1}\left(e^{\pi x/2}\right)-1,\quad x\geq 0,\nonumber\\\label{inverse}
\end{eqnarray}
and zero otherwise, respectively. Here, $\sinh$ denotes the hyperbolic sine function and $\mbox{sech}$  the hyperbolic secant function. Although not with this name, the pdf \eqref{pdfhshd} has been considered in the statistical literature by \cite{baten_1934} and \cite[p.148]{johnsonetal_1995}. Note that the cdf is related to the Gudermannian function \citetext{(\citealp{mcculley_1957}}, $\phi$, which is defined as
\begin{eqnarray*}
\phi=\mbox{gd}\,\psi =\int_0^{\psi} \mbox{sech}\,t\,dt=\tan^{-1}(\sinh \psi).
\end{eqnarray*}

Let us recall, too, the series expansions of the secant hyperbolic function \citetext{\citealp{becerraanddeeba_2008}, \citealp{weiandqi_2015}}
\begin{eqnarray}
\mbox{sech}(z) &=& 2 e^{-z}\sum_{k=0}^{\infty}(-1)^k e^{-2kz},\quad z>0,\label{sh}\\
\mbox{sech}(z) &=& \sum_{k=0}^{\infty}\frac{E_{2k} z^{2k}}{(2k)!},\quad |z|<\pi/2,\label{en}
\end{eqnarray}
which will be useful in the rest of the paper.

The next result gives the raw moments of the pdf \eqref{pdfhshd}.
\begin{theorem} Let the random variable $X$ have the half hyperbolic secant distribution with pdf \eqref{pdfhshd}. Then its raw moments are given by
\begin{eqnarray}
E_X(X^a)=\frac{\Gamma(1+a)}{\pi (2\pi)^a}\left[\xi\left(1+a,\frac{1}{4}\right)-\xi\left(1+a,\frac{3}{4}\right)\right],\quad a>-1,\label{moments}
\end{eqnarray}
where $\xi(s,m)=\sum_{k=0}^{\infty}(k+m)^{-s}$ gives the generalised zeta function.
\end{theorem}
\begin{proof} From \eqref{sh} we have
\begin{eqnarray}
E_X(X^a) &=& \int_0^{\infty}x^a \mbox{sech}\left(\frac{\pi x}{2}\right)\,dx=  2\int_0^{\infty}x^a e^{-\frac{\pi x}{2}}
\sum_{k=0}^{\infty}(-1)^k \exp\left(-2\frac{\pi x}{2}\right)\,dx\nonumber\\
& = & 2\sum_{k=0}^{\infty}(-1)^k\int_0^{\infty} x^a e^{-\frac{\pi x}{2}(1+2k)}\,dx= \frac{2\Gamma(1+a)}{\left(\frac{\pi}{2}\right)^{a+1}}\sum_{k=0}^{\infty} \frac{(-1)^k}{(1+2k)^{a+1}},\label{GG}
\end{eqnarray}
from which, taking into account that
\begin{eqnarray*}
\sum_{k=0}^{\infty} \frac{(-1)^k}{(1+2k)^{a+1}} &=& \sum_{k=0}^{\infty}\frac{1}{(1+4k)^{a+1}}-
\sum_{k=0}^{\infty}\frac{1}{(3+4k)^{a+1}}\\
&=& \sum_{k=0}^{\infty}\frac{4^{-(a+1)}}{(k+1/4)^{a+1}}-
\sum_{k=0}^{\infty}\frac{4^{-(a+1)}}{(k+3/4)^{a+1}}\\
&=& \frac{1}{2^{2+2 a}}\left[\xi\left(1+a,\frac{1}{4}\right)-\xi\left(1+a,\frac{3}{4}\right)\right]
\end{eqnarray*}
and after simple algebra, the result is obtained.
\end{proof}

Note that from \eqref{GG} by taking $a=0$ and taking into account that \eqref{pdfhshd} is a pdf, it is straightforward to see that
\begin{eqnarray*}
\frac{\pi}{4}=\sum_{k=0}^{\infty}\frac{(-1)^k}{1+2k},
\end{eqnarray*}
which is the Leibniz formula for $\pi$.

A few integral representations of Euler's numbers have been reported previously. One was provided by \cite{beesley_1969} and is given by
\begin{eqnarray*}
E_{2r}=i^{-2r}\left(\frac{2}{\pi}\right)^{2r+1}\int_0^{\infty}\frac{\log^{2r}(x)}{1+x^2}\,dx,\quad r=1,2,\dots
\end{eqnarray*}
where $i=\sqrt{-1}$.

Since the even moments of the pdf \eqref{pdfhshd} are related to Euler's numbers, two new integral representations can be obtained. To do so, first we need the following result.
\begin{theorem} It holds that
\begin{eqnarray*}
E_{2r}&=& (-1)^r E_X(X^{2r}),\quad r=1,2,\dots
\end{eqnarray*}
\end{theorem}
\begin{proof} By using the representation given in \eqref{sh} it is evident that
\begin{eqnarray*}
(-1)^rE_X(X^{2r})= (-1)^r2\left(\frac{2}{\pi}\right)^{1+2r}\sum_{k=0}^{\infty}(-1)^k\frac{\Gamma(1+2r)}{(1+2k)^{1+2r}},\quad r=1,2,\dots
\end{eqnarray*}

However, this corresponds to a series representation of Euler's numbers, $E_{2r}$.
\end{proof}

From the above, we can now express the following integral representation of Euler's numbers, using \eqref{rm} together with the inverse of the cdf given in \eqref{inverse}.
\begin{eqnarray*}
E_{2r} &=& (-1)^r\int_0^{\infty}x^{2r}\mbox{sech}\left(\frac{\pi x}{2}\right)\,dx = (-1)^r\int_0^{1} \left[\frac{2}{\pi}\log\left(\tan\left(\frac{\pi}{4}(x+1)\right)\right)\right]^{2r}\,dx\\
&=& (-1)^r\int_0^{1} \left[\frac{2}{\pi}\sinh^{-1}\left(\tan\left(\frac{x \pi}{2}\right)\right)\right]^{2r}\,dx.
\end{eqnarray*}
The following lemma is an important precursor of the result obtained in the corollary.
\begin{lemma}
It holds that
\begin{eqnarray}
\sum_{k=0}^{\infty}\frac{(-1)^k k}{(1+2k)^3}=\frac{1}{64}(32{\bf G}-\pi^3).\label{lema}
\end{eqnarray}
\end{lemma}
\begin{proof} We have
\begin{eqnarray*}
\sum_{k=0}^{\infty}\frac{(-1)^k k}{(1+2k)^3} &=& \frac{1}{2}\left[\sum_{k=0}^{\infty}\frac{(-1)^k }{(1+2k)^2}-\sum_{k=0}^{\infty}\frac{(-1)^k }{(1+2k)^3}\right]\\
&=& \frac{1}{2}{\bf G}-\frac{1}{128}\left[\xi(3,1/4)-\xi(3,3/4)\right]=\frac{1}{2}{\bf G}-\frac{\pi^3}{64}.
\end{eqnarray*}

Hence the result.
\end{proof}

The mean and variance of the pdf given in \eqref{pdfhshd} can be written in terms of Catalan's constant, as we will see in the next result.
\begin{corollary} The mean and variance of the pdf given in \eqref{pdfhshd} are given by
\begin{eqnarray}
E_X(X) &=& \mu_X=8\,{\bf G}/\pi^2,\label{mean}\\
var_X(X) &=& 1-64\,{\bf G}^2/\pi^4=(1-\mu_X)(1+\mu_X).\nonumber
\end{eqnarray}
\end{corollary}
\begin{proof}
The mean is obtained directly from \eqref{GG} by taking $a=1$ and taking into account \eqref{catalan}. To obtain the variance, we compute from \eqref{GG} the second raw moment, which is given by
\begin{eqnarray*}
E_X(X^2) &=& \frac{32}{\pi^3}\sum_{k=0}^{\infty}\frac{(-1)^k(1+2k-2k)}{(1+2k)^3}=\frac{32}{\pi^3}\left[{\bf G}-2\sum_{k=0}^{\infty} \frac{(-1)^k k}{(1+2k)^3}\right]\\
&=& \frac{32}{\pi^3}\left[{\bf G}-\frac{2}{64}(32{\bf G}-\pi^3)\right]=1,
\end{eqnarray*}
where we have incorporated \eqref{lema}. Now, the variance is computed from its definition, $var_X(X)=E_X(X^2)-\left(E_X(X)\right)^2$.
\end{proof}

Since the variance is positive, it follows that ${\bf G}<\pi^2/8$. Additionally, from \eqref{mean} and \eqref{m1} we can establish the following,
\begin{eqnarray}
{\bf G} &=& \frac{\pi^2}{8}\int_0^{\infty}x\, \mbox{sech}\left(\frac{\pi x}{2}\right)\,dx.\label{rep}\\
&=& \frac{\pi^2}{8}\int_0^{\infty}\left[1-\frac{2}{\pi}\tan^{-1}\left(\sinh\left(\frac{\pi x}{2}\right)\right)\right]\,dx\nonumber\\
&=&
\frac{\pi^2}{4}\int_0^{\infty}\left[1-\frac{2}{\pi}\tan^{-1}\left(e^{\frac{\pi x}{2}}\right)\right]\,dx.\nonumber
\end{eqnarray}

Furthermore, again from \eqref{mean} and by using \eqref{m2} we also establish
\begin{eqnarray*}
{\bf G} = \frac{\pi}{4}\int_0^1  \log\left(\tan\left(\frac{\pi}{4}(x+1)\right)\right)= \frac{\pi}{4}\int_0^1  \sinh^{-1}\left(\tan\left(\frac{\pi x}{2}\right)\right)\,dx.
\end{eqnarray*}

The expression given in \eqref{rep} is a reparameterisation of (5) provided in \cite{bradley_2001}.

If we introduce location and scale parameters in the pdf \eqref{pdfhshd}, we obtain another representation of Catalan's constant, which now depends on two parameters $\tau>0$ and $\sigma>0$, taking the form,
\begin{eqnarray*}
{\bf G} = \frac{1}{8}\left(\frac{\pi}{\sigma}\right)^2\left[\int_{\tau}^{\infty} x \,\mbox{sech}\left(\frac{\pi(x-\tau)}{2\sigma}\right)\,dx-\tau\right].
\end{eqnarray*}

By making use of representation \eqref{sh}, the moment generating function (mgf) can be represented as
\begin{eqnarray*}
M_X(t) &=& E_X(e^{t X})=2\sum_{k=0}^{\infty}(-1)^k\int_0^{\infty}e^{-x\left[\frac{\pi}{2}(1+2k)-t\right]}\,dx\\
&=& \frac{4}{\pi}\sum_{k=0}^{\infty}\frac{(-1)^k}{1+2k-2t/\pi}=\frac{2}{\pi}\Phi\left(-1,1,\frac{1}{2}-\frac{t}{\pi}\right),
\end{eqnarray*}
where $\Phi(z,s,a)=\sum_{k=0}^{\infty}z^k/(k+a)^s$ is the Lerch transcendent function.

By computing the derivative of $M_X(t)$ and evaluating this in $t=0$ we derive the mean $\mu_X$ {and hence}
\begin{eqnarray*}
{\bf G}= \frac{1}{16}\left[\psi^{\prime}(1/4)-\psi^{\prime}(3/4)\right],
\end{eqnarray*}
where $\psi^{\prime}(z)$ is the first derivative of the digamma function, which is a previously known result. It is also known (\cite{alzerandchoi_2017}) that
\begin{eqnarray*}
{\bf G}=\frac{1}{16}\left[\xi(2,1/4)-\xi(2,3/4)\right],
\end{eqnarray*}
Accordingly, it holds that
\begin{eqnarray*}
\psi^{\prime}(1/4)-\psi^{\prime}(3/4)=\xi(2,1/4)-\xi(2,3/4).
\end{eqnarray*}

Using Jensen's inequality it can be seen that $E(X^a)\geq (E(X))^a$ if $a<0$ or $a>1$ and $E(X^a)\leq (E(X))^a$ for $0<a<1$. Now, by using
\eqref{moments} and \eqref{mean} additional bounds for
Catalan's constant can be obtained.
\begin{theorem} Let $X_1$ and $X_2$ be independent random variables following the pdf \eqref{pdfhshd}. Then, the pdf of the random variable $Z=X_1+X_2$, which is the convolution of $X_1$ and $X_2$, is given by
\begin{eqnarray}
f_Z(z)=\frac{4}{\pi}\mbox{\normalfont csch}\left(\frac{\pi z}{2}\right)\log\left[\cosh\left(\frac{\pi z}{2}\right)\right],\quad z>0,\label{convolution}
\end{eqnarray}
where $\mbox{\normalfont csch}(z)$ gives the hyperbolic cosecant of $z$ and $\cosh(z)$ the hyperbolic cosine of $z$.
\end{theorem}
\begin{proof} {Using the convolution formula we have}
\begin{eqnarray}
f_Z(z) &=& \int_0^z f_X(x)f_X(z-x)\,dx=4\int_0^z \frac{dx}{\left(e^{\frac{\pi x}{2}}+e^{-\frac{\pi x}{2}}\right)\left(e^{\frac{\pi(z-x)}{2}}+e^{-\frac{\pi(z-x)}{2}}\right)}\nonumber\\
&=&\frac{8e^{\frac{\pi z}{2}}}{\pi}\int_1^{e^{\frac{\pi z}{2}}}\frac{u\,du}{(u^2+1)(u^2+e^{\pi z})}\nonumber\\
&=& \frac{8}{\pi}\frac{e^{\frac{\pi z}{2}}}{e^{\pi z}-1}\int_1^{e^{\frac{\pi z}{2}}}\left(\frac{u}{u^2+1}-\frac{u}{u^2+e^{\pi z}}\right)\,du\nonumber\\
&=& \frac{4}{\pi}\frac{e^{\frac{\pi z}{2}}}{e^{\pi z}-1}\left[\log\left(\frac{1+e^{\pi z}}{2e^{\pi z}}\right)-\log\left(\frac{2}{1+e^{\pi z}}\right)\right]\nonumber\\
&=& \frac{4}{\pi}\frac{2}{e^{\frac{\pi z}{2}}-e^{-\frac{\pi z}{2}}}\log\left(\frac{1+e^{\pi z}}{2e^{\frac{\pi z}{2}}}\right).\label{bas}
\end{eqnarray}

Hence the result.
\end{proof}

Now, given that $E_Z(Z)=E_{X_1}(X_1)+E_{X_2}(X_2)=16{\bf G}/\pi^2$, we can perform the following representation of Catalan's constant,
\begin{eqnarray*}
{\bf G} = \frac{\pi}{4}\int_0^{\infty}z\,\mbox{csch}\left(\frac{\pi z}{2}\right)\log\left[\cosh\left(\frac{\pi z}{2}\right)\right]\,dz.
\end{eqnarray*}

Furthermore, we can directly prove that $E_Z(Z^2)=2\left(1+\frac{64{\bf G}^2}{\pi^4}\right)$. In consequence, we also have the following representation of Catalan's constant,
\begin{eqnarray*}
{\bf G}^2=\frac{\pi^4}{64}\left(\frac{2}{\pi}\int_0^{\infty}z^2\,\mbox{csch}\left(\frac{\pi z}{2}\right)\log\left[\cosh\left(\frac{\pi z}{2}\right)\right]\,dz-1\right).
\end{eqnarray*}

The next result is derived {from the fact that} \eqref{convolution} is a genuine pdf and provides a new representation for  $\xi(2)=\sum_{k=1}^{\infty}k^{-2}=\pi^2/6$.
\begin{theorem} It holds that
\begin{eqnarray*}
\xi(2)=\sum_{k=1}^{\infty}\frac{1}{k^2}=\frac{4}{3}\left[\frac{\pi^2}{16}+\sum_{k=0}^{\infty}\frac{\Phi(-1,1,k+3/2)}{1+2k}\right],
\end{eqnarray*}
where $\xi(s)=\sum_{k=1}^{\infty}k^{-s}$ is the Riemann zeta function.
\end{theorem}
\begin{proof} Note that from \eqref{bas} and using the series representation of the logarithm function and the fact that $(1-e^{-\pi z})^{-1}=\sum_{k=0}^{\infty}e^{-2\pi z k}$ it is obtained that
\begin{eqnarray*}
1 &=& \int_0^{\infty}f_Z(z)\,dz=
\int_0^{\infty} \frac{8}{\pi}\left\{\frac{e^{-\pi z/2}}{1-e^{-\pi z}}\left[\log(1+e^{-\pi z})-\log 2+\frac{\pi z}{2}\right]\right\}\,dz\\
&=&
\frac{8}{\pi}\sum_{k=0}^{\infty}\int_0^{\infty}e^{-\frac{\pi z}{2}(1+2k)}\left\{\sum_{n=1}^{\infty}(-1)^{n+1}\frac{e^{-\pi n z}}{n}-\log 2+\frac{\pi z}{2}\right\}\,dz\\
&=& \frac{16}{\pi^2}\sum_{k=0}^{\infty}\left\{\sum_{n=1}^{\infty}\frac{(-1)^{n+1}}{n(1+2(k+n))}-\frac{\log 2}{1+2k}+\frac{1}{(1+2k)^2}\right\}.
\end{eqnarray*}

On the other hand,
\begin{eqnarray*}
\sum_{n=1}^{\infty}\frac{(-1)^{n+1}}{n(1+2(k+n))} &=& \frac{1}{1+2k}\sum_{n=1}^{\infty}(-1)^{n+1}\left(\frac{1}{n}-\frac{1}{n+k+1/2}\right)\\
&=& \frac{1}{1+2k}\left[\log 2-\sum_{r=1}^{\infty}\frac{(-1)^{r}}{r+k+3/2}\right]\\
&=& \frac{\log 2-\Phi(-1,1,k+3/2)}{1+2k}.
\end{eqnarray*}

Thus,
\begin{eqnarray}
\frac{16}{\pi^2}\left\{\sum_{k=0}^{\infty}\frac{1}{(1+2k)^2}-\sum_{k=0}^{\infty}\frac{\Phi(-1,1,k+3/2)}{1+2k} \right\}=1.\label{bas1}
\end{eqnarray}

Now, since $\sum_{k=0}^{\infty}(1+2k)^{-2}=(3/4)\xi(2)$ \citetext{see for instance \citealp{holst_2013}}, by replacing this in \eqref{bas1} the desired result is readily obtained.
\end{proof}

It is well-known that $\sum_{k=0}^{\infty}(1+2k)^{-2}=\pi^2/8$ \citetext{see for example, \citealp{holst_2013} and \citealp{jamesonandlord_2017}}, and so it holds that
\begin{eqnarray*}
\sum_{k=0}^{\infty}\frac{\Phi(-1,1,k+3/2)}{1+2k}=\frac{\pi^2}{16}.\label{bas2}
\end{eqnarray*}

\subsection{Additional representations of Catalan's constant}
In this section, we obtain new representations of Catalan's constant, from monotonic transformations of the half hyperbolic secant random variable. For example, by assuming in \eqref{pdfhshd} the change of variable $X=-\log Y$ we can obtain a new pdf with support in $(0,1)$ given by
\begin{eqnarray*}
f_Y(y)= \frac{1}{y}\,\mbox{sech}\left(\frac{\pi}{2}\log y\right),\quad 0<y<1.
\end{eqnarray*}

Thus, by construction we have $E_Y(\log Y)=E_X(X)=8{\bf G}/{\pi^2}$ and therefore
\begin{eqnarray*}
{\bf G} = -\frac{\pi^2}{8}\int_0^1\frac{\log y}{y}\,\mbox{sech}\left(\frac{\pi}{2}\log y\right)\,dy.
\end{eqnarray*}

Similarly, by making the change of variable $X=Y/(1-Y)$ from \eqref{pdfhshd} we obtain the pdf
\begin{eqnarray*}
f_Y(y)=\frac{1}{(1-y)^2}\,\mbox{sech}\left(\frac{\pi y}{2(1-y)}\right),\quad 0<y<1
\end{eqnarray*}
and again, since $E_Y(Y/(1-Y))=E_X(X)$ we have
\begin{eqnarray}
{\bf G} = \frac{\pi^2}{8}\int_0^1\frac{y}{(1-y)^3}\,\mbox{sech}\left(\frac{\pi y}{2(1-y)}\right)\,dy.\label{trapezoidal}
\end{eqnarray}

The expression given in \eqref{trapezoidal} can approximate Catalan's constant using the trapezoidal rule. In fact, when the domain $(0,1)$ is discretised into $N$ equally spaced panels, we have
\begin{eqnarray*}
{\bf G} \approx \frac{\pi^2}{8 N}\left[\frac{f(x_N)+f(x_0)}{2}+\sum_{k=1}^{N-1}f(x_k)\right],
\end{eqnarray*}
where
\begin{eqnarray*}
f(x)=\frac{1}{(1-x)^3}\,\mbox{sech}\left(\frac{\pi x}{2(1-x)}\right),
\end{eqnarray*}
which provides an asymptotic error estimate given by
\begin{eqnarray*}
\mbox{Error } = -\frac{(x_N-x_0)^2}{12 N^2}\left[f^{\prime}(x_N)-f^{\prime}(x_0)\right]+O(N^{-3}),\quad N\to \infty.
\end{eqnarray*}

This procedure can also be applied to \eqref{ir} or another representation through an integral when the support of the distribution is bounded.

By taking the change of variable $X=1/Y$, we obtain the pdf
\begin{eqnarray}
f_Y(y)=\frac{1}{y^2}\,\mbox{sech}\left(\frac{\pi}{2 y}\right),\quad y>0,\label{ihs}
\end{eqnarray}
and $E_Y(1/Y)=E_X(X)$. {In consequence we have}
\begin{eqnarray}
{\bf G} = \frac{\pi^2}{8}\int_0^{\infty}\frac{1}{y^3}\,\mbox{sech}\left(\frac{\pi}{2 y}\right)\,dy.\label{mihs}
\end{eqnarray}

Finally, by taking the change of variable $X=\log(e^y+1)$ we have the pdf given by
\begin{eqnarray*}
f_Y(y)=\frac{e^y}{e^y+1}\,\mbox{sech}\left(\frac{\pi}{2}\log(e^y+1)\right),\quad -\infty<y<\infty.
\end{eqnarray*}

Now, {since} $E_Y(\log(e^{y}+1))=E_X(X)$ we obtain
\begin{eqnarray*}
{\bf G}=\frac{\pi^2}{8}\int_{-\infty}^{\infty}\frac{e^y\log(e^y+1)}{e^y+1}\,\mbox{sech}\left(\frac{\pi}{2}\log(e^y+1)\right)\,dy.
\end{eqnarray*}

\section{Catalan's constant from measures of inequality}\label{s3}
In studies of income distribution, the Lorenz curve (LC), which informs about the cumulative proportion of income held by the bottom $p$ percent of the population, is an instrument of crucial importance. Offering a complete picture of the concentration in the distribution, this graphical tool has become popular in areas beyond {economics and} social sciences, extending to fields such as bibliometrics and physics.

Let $X$ be a random variable with support in the subset of the positive real number, say ${\cal L}$, with cdf $F_X(x)$, inverse distribution function $F_X^{-1}(x)=\sup\left\{y: F_X(y)\leq x\right\}$ { and positive} finite expectation $\mu_X$. Then, the
LC of $X$ is defined (see \cite{gastwith_1971}) as
\begin{equation}\label{defbas}
L_X(p)=\frac{1}{\mu_X}\int_0^pF_X^{-1}(y)\,dy,\quad 0\le p\le 1.
\end{equation}

The result provides the LC associated with the pdf \eqref{pdfhshd}.
\begin{theorem}
The Lorenz curve associated with the pdf given in \eqref{pdfhshd} is given by
\begin{eqnarray}
L_X(p)= \frac{1}{\mu_X}\left[(1+p)F_X^{-1}(p)-\frac{2 z(p)}{\pi^2}\Phi(-z(p)^2,2,1/2)\right]+1,\label{lchhsd}
\end{eqnarray}
where
\begin{eqnarray*}
F_X^{-1}(p) &=& \frac{2}{\pi}\log\left[z(p)\right],\\
z(p) &=& \tan\left(\frac{\pi}{4}(p+1)\right)
\end{eqnarray*}
and $\mu_X$ is given in \eqref{mean}.
\end{theorem}
\begin{proof}
From \eqref{defbas} we have that
\begin{eqnarray*}
\mu_X L_X(p) &=& \int_0^p F_X^{-1}(t)\,dt=p F_X^{-1}(p)-\int_{F_X^{-1}(0)}^{F_X^{-1}(p)}F_X(t)\,dt,
\end{eqnarray*}
where we have used the formula (1) in \cite{anderson_1970} which describes the integration of the inverse function. From this, we obtain
\begin{eqnarray*}
\mu_X L_X(p) &=& pF_X(p)-\int_0^{F_X^{-1}(p)}\left[\frac{4}{\pi}\tan^{-1}\left(e^{\frac{\pi t}{2}}\right)-1\right]\,dt\\
&=& pF_X(p)-\int_0^{F_X^{-1}(p)}\left[\frac{4}{\pi}\sum_{k=0}^{\infty}(-1)^k\frac{e^{\pi (1+2k) t/2}}{1+2k}-1\right]\,dt,
\end{eqnarray*}
where we use the series representation of the $\tan^{-1}$ function provided in \eqref{srt}. Thus,
\begin{eqnarray*}
\mu_X L_X(p) &=& pF_X(p)-\frac{4}{\pi}\sum_{k=0}^{\infty}\frac{(-1)^k}{1+2k}\int_0^{F_X^{-1}(p)}e^{\pi (1+2k) t/2}\,dt+F_X^{-1}(p)\\
&=& pF_X(p)-\frac{4}{\pi}\sum_{k=0}^{\infty}\frac{(-1)^k}{1+2k}\left.\frac{2e^{\pi (1+2k) t/2}}{\pi(1+2k)}\right|_0^{F_X^{-1}(p)}+F_X^{-1}(p)\\
&=& pF_X(p)-\frac{8}{\pi^2}\sum_{k=0}^{\infty}\frac{(-1)^k}{(1+2k)^2}\left[e^{\frac{\pi}{2}(1+2k)F_X^{-1}(p)}-1\right]+F_X^{-1}(p).
\end{eqnarray*}

Now, from $F_X^{-1}(p)=(2/\pi)\log(z(p))$ we have
\begin{eqnarray*}
\mu_X L_X(p) = pF_X(p)-\frac{8}{\pi^2}\sum_{k=0}^{\infty}\frac{(-1)^k}{(1+2k)^2}z(p)^{1+2k}+\frac{8{\bf G}}{\pi^2}+F_X^{-1}(p).
\end{eqnarray*}

Finally, taking into account that
\begin{eqnarray}
\sum_{k=0}^{\infty}\frac{(-1)^k}{(1+2k)^2}z(p)^{1+2k}=\frac{z(p)}{4}\sum_{k=0}^{\infty}\frac{(-z(p)^2)^k}{(k+1/2)^2}=\frac{z(p)}{4}\Phi(-z(p)^2,2,1/2),\nonumber\\\label{lerch}
\end{eqnarray}
the result follows after simple algebra.
\end{proof}

The Gini index \citetext{see among others, \citealp{kleiberandkotz_2003}} is one of the most important inequality measures used in economics. This index is calculated as twice the area between the egalitarian line and the LC. Thus, if $X$ is a random variable in ${\cal L}$ with Lorenz curve $L_X(p)$, a formula for the Gini index, $G(X)$, is
\begin{equation}
G(X)=2\int_0^1 \left[p- L_X(p)\right] dp=1-2\int_0^1L_X(p)\,dp.
\end{equation}

\begin{theorem}
The Gini index for the LC defined in \eqref{lchhsd} is given by
\begin{eqnarray}
G(X)=\frac{7\xi(3)}{2\pi{\bf G}}-1.\label{ginii}
\end{eqnarray}
\end{theorem}
\begin{proof}
The Gini index can be computed as
\begin{eqnarray*}
G(X) &=& 2\int_{0}^{1}pL_X'(p)dp-1=\frac{2}{\mu_X}\int_{0}^{1}pF_X^{-1}(p)dp-1\\
& = & \frac{1}{\mu_X}\frac{4}{\pi}\int_{0}^{1}p\log\left(\tan\left(\frac{\pi(p+1)}{4}\right)\right)dp-1.
\end{eqnarray*}

Now, let
\begin{eqnarray*}\label{fgini}
I &=& \int_{0}^{1}p\log\left(\tan\left(\frac{\pi(p+1)}{4}\right)\right)dp=\frac{4}{\pi}\int_{\pi/4}^{\pi/2}\left(\frac{4x}{\pi}-1\right)\log(\tan x)\,dx\\
&=&  -\frac{8}{\pi}\sum_{k=0}^{\infty}\int_{\pi/4}^{\pi/2}\left(\frac{4x}{\pi}-1\right)\frac{\cos(2(1+2k)x)}{1+2k}\,dx,\quad x\in\left(0,\frac{\pi}{2}\right),
\end{eqnarray*}
where we use the series representation of the $\log(\tan(\cdot))$ function provided in \cite{elaissaouiandguennoun_2019}. Some straightforward computation then produces
\begin{eqnarray*}
I &=& \frac{4}{\pi^2}\sum_{k=0}^{\infty}\frac{2\cos(2k\pi)-2\sin(k\pi)+(1+2k)\pi\sin(2k\pi)}{(1+2k)^3}\\
&=& \frac{8}{\pi^2}\sum_{k=0}^{\infty}\frac{1}{(1+2k)^3}=\frac{7\xi(3)}{\pi^2}.
\end{eqnarray*}

The result then follows after simple algebra.
\end{proof}

Now, since $0<G(X)<1$, it is straightforward to obtain
\begin{eqnarray*}
\frac{7\xi(3)}{4\pi}<{\bf G}<\frac{7\xi(3)}{2\pi}.
\end{eqnarray*}

By using \eqref{ginii} together with the definition of the Gini index, we directly obtain a double-integral representation {for} {\bf G}, given by
\begin{eqnarray*}
{\bf G}=\frac{7\xi(3)}{4\pi}+\frac{\pi}{4}\int_0^1\int_0^p\log\left(\tan\left(\frac{\pi}{4}(t+1)\right)\right)\,dt\,dp.
\end{eqnarray*}

\cite{eliazarandsokolov_2010} suggested that the Gini index can also be computed using the formula
\begin{eqnarray}
G(X)=\frac{1}{\mu_X}\int_0^{\infty}F_X(x)(1-F_X(x))\,dx.\label{othergini}
\end{eqnarray}

Then, from \eqref{ginii} and \eqref{othergini} together with \eqref{inverse}, we have the following representations of Catalan's constant,
\begin{eqnarray*}
{\bf G} &=& \frac{7\xi(3)}{2\pi}-\frac{\pi}{4}\int_0^{\infty}\tan^{-1}\left(\mbox{sinh}\left(\frac{\pi x}{2}\right)\right)\left[1-\frac{2}{\pi}\tan^{-1}\left(\mbox{sinh}\left(\frac{\pi x}{2}\right)\right)\right]\,dx,\\
{\bf G} &=& \frac{7\xi(3)}{2\pi}-\frac{\pi^2}{4}\int_0^{\infty}\left[\frac{4}{\pi}\tan^{-1}\left(e^{\pi x/2}\right)-1\right]\left[1-\frac{2}{\pi}\tan^{-1}\left(e^{\pi x/2}\right)\right]\,dx.
\end{eqnarray*}

The Theil index, another measure of economic inequality, is defined by
\begin{eqnarray*}
T_0(X)=-E_X(\log(X/\mu_X))=-E_X(\log X)+\log \mu_X,
\end{eqnarray*}
where $\mu_X$ is as given in \eqref{mean}.

By using the result established by \cite{reynoldsandstauffer_2020}, according to which
\begin{eqnarray*}
\int_0^{\infty} \log y\, \mbox{sech}\, y\,dy=\pi\log\left(\frac{\sqrt{2\pi}\Gamma(3/4)}{\Gamma(1/4)}\right)
\end{eqnarray*}
we have
\begin{eqnarray*}
\int_0^{\infty} \log y\, \mbox{sech}\, y\,dy &=& \frac{\pi}{2}\int_0^{\infty}\left[\log\left(\frac{\pi}{2}\right)+\log x\right]\,\mbox{sech}\left(\frac{\pi x}{2}\right)\,dx\\
&=&
\frac{\pi}{2}\left[\log\left(\frac{\pi}{2}\right)+\int_0^{\infty} \log x\,\mbox{sech}\,\left(\frac{\pi x}{2}\right)\right]\,dx.
\end{eqnarray*}

and hence
\begin{eqnarray}
E_X(\log X)) = 2\log\left(\frac{\sqrt{2\pi}\Gamma(3/4)}{\Gamma(1/4)}\right)-\log\left(\frac{\pi}{2}\right)=\log(8\pi^2)-4\log\Gamma(1/4)\nonumber\\\label{elog}
\end{eqnarray}
and therefore,
\begin{eqnarray*}
T_0(X)=\log {\bf G}+4\log\left(\frac{\Gamma(1/4)}{\pi}\right),
\end{eqnarray*}
where we have used the {identity} $\Gamma(1/4)\Gamma(3/4)=\pi\sqrt{2}$.

\begin{theorem} It holds that
\begin{eqnarray*}
\sum_{k=0}^{\infty}(-1)^k\frac{ \log(1+2k)}{1+2k}=\frac{\pi}{4}\left[-\gamma-\log(4\pi^3)+4\log\Gamma(1/4)\right],
\end{eqnarray*}
where $\gamma\approx 0.577216 $ is the Euler-Mascheroni constant.
\end{theorem}
\begin{proof} {We define} the function $g(a)=E_X(X^a)$. Then, from \eqref{GG} we have
\begin{eqnarray*}
g^{\prime}(a) &=& E_X(X^a \log X)= g(a)\left\{\psi(1+a)-\log\left(\frac{\pi}{2}\right)\right.\\
&&\left. -\left[{\sum_{k=0}^{\infty}(-1)^k\frac{1}{(1+2k)^{a+1}}}\right]^{-1}\sum_{k=0}^{\infty}(-1)^k\frac{ \log(1+2k)}{(1+2k)^{a+1}}\right\},
\end{eqnarray*}
where $\psi(z)=(d/dz)\log\Gamma(z)$ is the digamma function. Thus,
\begin{eqnarray*}
g^{\prime}(0)=E_X(\log X)=\psi(1)-\log\left(\frac{\pi}{2}\right)-\frac{4}{\pi}\sum_{k=0}^{\infty}(-1)^k\frac{ \log(1+2k)}{1+2k}.
\end{eqnarray*}

Now, by replacing $E_X(\log X)$ by \eqref{elog}, the result is obtained directly, taking into account that $\psi(1)=-\gamma$.
\end{proof}

An interesting but less well-known index of inequality is the
Pietra index, given by the proportion of total income that would
need to be reallocated across the population to achieve perfect
income equality. This {index is defined as},
\begin{eqnarray*}
P=\max_{0\leq p\leq 1}\left[p-L_{\alpha}(p)\right]=\frac{1}{2\mu}E|X-\mu|
\end{eqnarray*}
and corresponds to the maximal vertical deviation between the
Lorenz curve and the egalitarian line \citetext{\citealp{arnold_1986}}.

\begin{theorem}
The Pietra index associated with the pdf given in \eqref{pdfhshd} is achieved in
\begin{eqnarray}
p=\frac{4}{\pi}\tan^{-1}\left(e^{\frac{4{\bf G}}{\pi}}\right)-1.\label{pietra1}
\end{eqnarray}
and therefore is given by
\begin{eqnarray}
P(X)=\frac{1}{4{\bf G}}\Phi\left(-e^{\frac{8{\bf G}}{\pi}},2,\frac{1}{2}\right)e^{\frac{4{\bf G}}{\pi}}-2.\label{pietra2}
\end{eqnarray}
\end{theorem}
\begin{proof}
After computing the derivative of $L_X(p)-p$ and following some straightforward algebra, we see that
\begin{eqnarray*}
\frac{d}{dp}(L_X(p)-p)=\frac{\pi}{4{\bf G}}\log\left(z(p)\right)-1.
\end{eqnarray*}

Now, by equating this equation to zero and isolating $p$, we obtain \eqref{pietra1}. Finally, by substituting this value in $L(p)-p$, we have \eqref{pietra2}.
\end{proof}

\cite{eliazarandsokolov_2010} suggest an alternative way to derive the Pietra index, via the formula
\begin{eqnarray}
P(X)=\frac{1}{\mu_X}\int_{\mu_X}^{\infty}(1-F_X(x))\,dx.\label{otherpietra}
\end{eqnarray}

Now, by applying \eqref{otherpietra} together with \eqref{pietra2} we have
\begin{eqnarray*}
{\bf G} &=& \frac{1}{16}\left\{2\Phi\left(-e^{\frac{8{\bf G}}{\pi}},2,\frac{1}{2}\right)e^{\frac{4{\bf G}}{\pi}}-\pi^2\int_{8{\bf G}/\pi^2}^{\infty}
\left[1-\frac{2}{\pi}\tan^{-1}\left(\sinh\left(\frac{\pi x}{2}\right)\right)\right]\,dx\right\}\\
&=& \frac{1}{16}\left\{2\Phi\left(-e^{\frac{8{\bf G}}{\pi}},2,\frac{1}{2}\right)e^{\frac{4{\bf G}}{\pi}}-2\pi^2\int_{8{\bf G}/\pi^2}^{\infty}
\left[1-\frac{2}{\pi}\tan^{-1}\left(e^{\pi x/2}\right)\right]\,dx\right\}.
\end{eqnarray*}

\subsection{Other representations of Catalan's and Ap\'ery's constants}

In this section we obtain new representations of Catalan's and Ap\'ery's constants, based on expressions of the Gini index. For the first of these constants we have:

\begin{equation}\label{catalan11}
{\bf G}=\frac{7\xi(3)}{4\pi}+\frac{\pi^2}{16}\int_{0}^{\infty} \left(1-\frac{2}{\pi}\tan^{-1}\left(\sinh\left(\frac{\pi x}{2}\right)\right)\right)^2dx,
\end{equation}
and
\begin{equation}\label{catalan12}
{\bf G}=\frac{7\xi(3)}{2\pi}-\frac{\pi}{4}\int_{0}^{\infty}\tan^{-1}\left(\sinh\left(\frac{\pi x}{2}\right)\right)\left(1-\frac{2}{\pi}\tan^{-1}\left(\sinh\left(\frac{\pi x}{2}\right)\right)\right)dx.
\end{equation}
{The proof of (\ref{catalan11}) is based on the representation of the Gini index,
$$G(X)=1-\frac{1}{\mu_X}\int_{0}^{\infty}[1-F_X(x)]^2dx,$$
which can be found in \cite{arnold_2018}, equation (5.13). Expression \eqref{catalan12} is obtained using equations \eqref{ginii} and \eqref{othergini}.} \\

Ap\'ery's constant can be represented in these two ways,
\begin{equation}\label{appery11}
\xi(3)=\frac{\pi^3}{28}\int_{0}^{\infty}\left\{1-\left(
\frac{2}{\pi}\tan^{-1}\left(\sinh\left(\frac{\pi x}{2}\right)\right)
\right)^2\right\}dx,
\end{equation}
and
\begin{equation}\label{appery12}
\xi(3)=\frac{\pi^3}{14}\int_{0}^{\infty}x\,\mbox{sech}\left(\frac{\pi x}{2}\right)\frac{2}{\pi}\tan^{-1}\left(\sinh\left(\frac{\pi x}{2}\right)\right)dx.
\end{equation}
{Formulas (\ref{appery11}) and (\ref{appery12}) are obtained from (\ref{ginii}) and the alternative expressions for the Gini index,
$$G(X)=\frac{1}{\mu_X}\int_{0}^{\infty}[1-F(x)^2]dx-1,$$
and
$$G(X)=-1+\frac{2}{\mu_X}\int_{0}^{\infty}xf_X(x)F_X(x)dx.$$
This last expression corresponds to equation (5.15) in \cite{arnold_2018}.}

\section{Catalan's constant from mixtures of distributions}\label{s4}
Mixtures of  distributions play an important role in both the theory
and the practice of statistics. Suppose that $Y$ is a discrete or continuous random variable with mean $x>0$ and suppose that $x$ is random and follows the pdf given in \eqref{pdfhshd}. In this case, the pdf \eqref{pdfhshd} is referred to as the mixing density. Then, the unconditional probability function of $Y$ is given by $f_Y(Y)=E_X( f_Y(y|x))$ and the mean of this distribution is given by $E_Y(Y)=E_X(E_Y(Y|X))=8{\bf G}/\pi^2$. It then follows that
\begin{eqnarray*}
{\bf G}=\frac{\pi^2}{8}\int_0^{\infty}\int_Y y f_Y(y|x)f_X(x)\,dy\,dx=\frac{\pi^2}{8}\int_0^{\infty}\int_Y y f_Y(y|x)\,\mbox{sech}\left(\frac{\pi x}{2}\right)\,dy\,dx,
\end{eqnarray*}
when $Y$ is continuous and
\begin{eqnarray*}
{\bf G}=\frac{\pi^2}{8}\int_0^{\infty}\sum_{y=0}^{\infty} y f_Y(y|x)f_X(x)\,dx=\frac{\pi^2}{8}\int_0^{\infty}\sum_{y=0}^{\infty} y f_Y(y|x)\,\mbox{sech}\left(\frac{\pi x}{2}\right)\,dx,
\end{eqnarray*}
in the discrete case. {Some examples are now provided.}
\subsubsection*{Poisson} Suppose that the random variable $Y$ follows a Poisson distribution with mean $x>0$ and $x$ follows the pdf given in \eqref{pdfhshd}. Then the unconditional distribution of $Y$, whose probability function is not reproduced here, has a mean given by $E(Y)=8{\bf G}/\pi^2 $. Thus, we have the following representation of Catalan's constant,
\begin{eqnarray*}
{\bf G} &=& \frac{\pi^2}{8}\sum_{y=1}^{\infty}\frac{1}{(y-1)!}\int_0^{\infty}x^y \exp(-x)\,\mbox{sech}\left(\frac{\pi x}{2}\right)\,dx\\
&=& \frac{\pi^2}{2}\sum_{y=1}^{\infty}\sum_{k=0}^{\infty}\frac{(-1)^k 2^{y}y}{[2+\pi(1+2k)]^{y+1}},
\end{eqnarray*}
{for which we have used equation} \eqref{sh}.

\subsubsection*{Negative binomial}
Suppose the random variable $Y$ follows a negative binomial distribution with parameters $r>0$ and $p\in (0,1)$. Let us take the success probability as $p=r/(r+x)$ {and then} the mean of $Y$ is $x>0$. Assume now that $x$ follows the pdf given in \eqref{pdfhshd}. Then the unconditional distribution of $Y$ has a mean given by $E(Y)=8{\bf G}/\pi^2 $, from which we have the following representation of Catalan's constant
\begin{eqnarray*}
{\bf G} &=& \frac{\pi^2}{8}\sum_{k=1}^{\infty}\frac{(r+k-1)!r^r}{(r-1)!(k-1)!}\int_0^{\infty}\frac{x^k}{(r+x)^{r+k}}\,\mbox{sech}\left(\frac{\pi x}{2}\right)\,dx\\
&=& \frac{\pi^2}{4}\sum_{k=1}^{\infty}\frac{(r+k-1)!r^r}{(r-1)!(k-1)!}\int_0^{\infty}
\sum_{j=0}^{\infty}(-1)^j x^k (r+x)^{-r-k}e^{-\pi x(1+2j)/2}\,dx\\
&=& \frac{\pi^2}{4}\sum_{k=1}^{\infty}\frac{\Gamma(r+k)r}{\Gamma(r)\Gamma(k)}\int_0^{\infty}
\sum_{j=0}^{\infty}(-1)^j s^k (1+s)^{-r-k}e^{-\pi r s(1+2j)/2}\,ds\\
&=& \frac{\pi^2 }{4}\frac{r}{\Gamma(r)}\sum_{k=1}^{\infty} k\Gamma(r+k)\sum_{j=0}^{\infty}(-1)^j\,{\cal U}\left(k+1,2-r,\frac{r \pi}{2}(1+2j)\right),
\end{eqnarray*}

where
\begin{eqnarray*}
{\cal U}(a,b,z)=\frac{1}{\Gamma(a)}\int_0^{\infty} s^{a-1}(1+s)^{b-a-1}\exp(-z s)\,ds,\quad a,z>0,
\end{eqnarray*}
is the confluent hypergeometric function \citetext{see \citealp[p. 1085]{gradshteynetal_1994}}. By setting $r=1$, we have the special case of the geometric mixture {and we obtain}
\begin{eqnarray*}
{\bf G} =
\frac{\pi^2}{4} \sum_{k=1}^{\infty} k k!\sum_{j=0}^{\infty}(-1)^j\,{\cal U}\left(k+1,1,\frac{\pi}{2}(1+2j)\right).
\end{eqnarray*}


\subsubsection*{Exponential mixture}
Let $f(y|x)=x\exp(-x y)$ be an exponential distribution with the parameter $x>0$ and with the mean $E_Y(Y|x)=1/x$. Suppose now that $X$ follows the pdf given in
\eqref{ihs}. Then, by using the representation provided in \eqref{sh} and after some algebra {we obtain} the unconditional distribution of $Y$, which is given by
\begin{eqnarray*}
f_Y(y)=4\sum_{k=0}^{\infty}(-1)^k K_0(\sqrt{2\pi y(1+2k)}),\quad y>0,
\end{eqnarray*}
where $K_n(z)$ gives the modified Bessel function of the second kind. Then, the unconditional mean of $Y$ {is} $E_Y(Y)=E_X(1/X)$, and by using \eqref{mihs} together with the representation provided in \eqref{sh} we have
\begin{eqnarray*}
{\bf G} =
\frac{\pi^2}{2} \sum_{k=0}^{\infty} (-1)^k \int_0^{\infty} y\, K_0(\sqrt{2\pi y(1+2k})\,dy.
\end{eqnarray*}

\section{Further results}\label{s5}
Consider now the following pdf,
\begin{eqnarray}
f_X(x)=\frac{\pi}{2\tan^{-1}(\sinh(\pi/2))}\mbox{sech}\left(\frac{\pi x}{2}\right),\quad 0\leq x \leq 1,\label{npdf}
\end{eqnarray}
which is obtained from \eqref{pdfhshd} by restricting its {support} to the interval $[0,1]$. {The corresponding} cdf is given by
\begin{eqnarray*}
F_X(x)=\frac{\tan^{-1}(\sinh(\pi x/2)}{\tan^{-1}(\sinh(\pi/2))},\quad 0\leq x\leq 1.
\end{eqnarray*}

{Using} representation \eqref{en} and taking into account that \eqref{npdf} is a genuine pdf, it follows directly that
\begin{eqnarray}
\sum_{k=0}^{\infty}\frac{E_{2k}}{(1+2k)(2k)!}\left(\frac{\pi}{2}\right)^{2k}=\sum_{\substack{k=0 \\ k\,\mbox{\tiny even} }}^{\infty}\frac{E_{k}}{(1+k)k!}\left(\frac{\pi}{2}\right)^{k}=
\frac{2}{\pi}\tan^{-1}(\sinh(\pi/2)).\nonumber\\ \label{r1}
\end{eqnarray}

The next result provides the mean of \eqref{npdf}, which is also expressed in terms of Catalan's constant.
\begin{theorem} If $X$ follows the pdf \eqref{npdf} then its {mathematical expectation} is given by
\begin{eqnarray}
E_X(X) = \frac{4}{\pi\tan^{-1}(\sinh(\pi/2))}\left[{\bf G}-\frac{\pi}{2}\tan^{-1}(e^{-\pi/2})-\frac{e^{-\pi/2}}{4}\,
\Phi\left(-e^{-\pi},2,\frac{1}{2}\right)\right].\nonumber\\
\label{nmean}
\end{eqnarray}
\end{theorem}
\begin{proof} By using the representation \eqref{sh} and the fact that $\int_0^1 x e^{-ax}\,dx=[1-(1+a)]e^{-a}/a$ we have
\begin{eqnarray*}
E_X(X) &=& \frac{4}{\pi\tan^{-1}(\sinh(\pi/2))}\sum_{k=0}^{\infty}(-1)^k\frac{1-\left[1+\frac{\pi}{2}(1+2k)\right]e^{-\frac{\pi}{2}(1+2k)}}{(1+2k)^2}\\
&=& \frac{4}{\pi\tan^{-1}(\sinh(\pi/2))}\left\{{\bf G}-e^{-\pi/2}\left[\sum_{k=0}^{\infty}\frac{(-1)^ke^{-\pi k}}{(1+2k)^2}+\frac{\pi}{2}\sum_{k=0}^{\infty}(-1)^k\frac{e^{-\pi k}}{1+2k}\right]\right\}.
\end{eqnarray*}

Now, from \eqref{srt} we see that
\begin{eqnarray*}
\sum_{k=0}^{\infty}(-1)^k\frac{e^{-\pi k}}{1+2k}=e^{\pi/2}\tan^{-1}(e^{-\pi/2}),
\end{eqnarray*}
and from \eqref{lerch}
\begin{eqnarray*}
\sum_{k=0}^{\infty}\frac{(-1)^ke^{-\pi k}}{(1+2k)^2} = \frac{1}{4}\,
\Phi\left(-e^{-\pi},2,\frac{1}{2}\right),
\end{eqnarray*}
and after some algebra, the result follows.
\end{proof}

Thus, we have the following representation of Catalan's constant,
\begin{eqnarray}
{\bf G} = \frac{\pi^2}{8}\int_0^1 x\,\mbox{sech}\left(\frac{\pi x}{2}\right)\,dx+\frac{\pi}{2}\tan^{-1}(e^{-\pi/2})
+\frac{e^{-\pi/2}}{4}\,
\Phi\left(-e^{-\pi},2,\frac{1}{2}\right).\nonumber\\
\label{nmpdf}
\end{eqnarray}

Now, by using the representation given in \eqref{en}, we obtain
\begin{eqnarray}
E_X(X) &=& \frac{\pi}{2\tan^{-1}(\sinh(\pi/2))}\int_0^{1}x\sum_{k=0}^{\infty}\frac{E_{2k}}{(2k)!}\left(\frac{\pi x}{2}\right)^{2k}\,dx\nonumber\\
&=& \frac{\pi}{2\tan^{-1}(\sinh(\pi/2))}\sum_{k=0}^{\infty}\frac{E_{2k}}{(2k)!}\left(\frac{\pi}{2}\right)^{2k}\int_0^{1}x^{2k+1}\,dx\nonumber\\
&=& \frac{\pi}{4\tan^{-1}(\sinh(\pi/2))}\sum_{k=0}^{\infty}\frac{E_{2k}}{(1+k)(2k)!}\left(\frac{\pi}{2}\right)^{2k}.\label{G2}
\end{eqnarray}

Therefore, by equating \eqref{G2} with \eqref{nmean} we arrive at the following:
\begin{eqnarray*}
{\bf G}= \frac{\pi^2}{16}\sum_{k=0}^{\infty}\frac{E_{2k}}{(1+k)(2k)!}\left(\frac{\pi}{2}\right)^{2k}+\frac{\pi}{2}\tan^{-1}(e^{-\pi/2})
+\frac{e^{-\pi/2}}{4}\,
\Phi\left(-e^{-\pi},2,\frac{1}{2}\right).
\end{eqnarray*}

Furthermore, from \eqref{nmpdf}, taking into account \eqref{npdf} and again using \eqref{m1} and \eqref{m2}, we obtain the following representations of Catalan's constant
\begin{eqnarray*}
{\bf G} &=& \frac{\pi}{4}\left[\tan^{-1}\left(\mbox{sinh}\left(\frac{\pi}{2}\right)\right)-
\int_0^1\tan^{-1}\left(\mbox{sinh}\left(\frac{\pi x}{2}\right)\right)\,dx\right]\\
&&+\frac{\pi}{2}\tan^{-1}(e^{-\pi/2})) + \frac{e^{-\pi/2}}{4}\,
\Phi\left(-e^{-\pi},2,\frac{1}{2}\right),\\
{\bf G} &=& \frac{1}{2}\tan^{-1}\left(\mbox{sinh}\left(\frac{\pi}{2}\right)\right)\int_0^1 \mbox{sinh}^{-1}\left(\tan\left(x\tan^{-1}\left(\mbox{sinh}\left(\frac{\pi}{2}\right)\right)\right)\right)\,dx\\
&&+\frac{\pi}{2}\tan^{-1}(e^{-\pi/2})
+\frac{e^{-\pi/2}}{4}\,
\Phi\left(-e^{-\pi},2,\frac{1}{2}\right).
\end{eqnarray*}

Finally, it is known \citetext{\citealp{holst_2013}} that
\begin{eqnarray}
E_{2k}\left(\frac{\pi}{2}\right)^{2k}=\frac{4}{\pi}\sum_{j=0}^{\infty}(-1)^j\int_0^{\infty} x^{2k}\exp[-x(1+2j)]\,dx,\quad k\,\mbox{even},\label{r2}\\
E_{2k}\left(\frac{\pi}{2}\right)^{2k}=-\frac{4}{\pi}\sum_{j=0}^{\infty}(-1)^j\int_0^{\infty} x^{2k}\exp[-x(1+2j)]\,dx,\quad k\,\mbox{odd},\label{r3}
\end{eqnarray}
and therefore, from \eqref{r1}, \eqref{r2} and \eqref{r3} we have
\begin{eqnarray*}
\sum_{\substack{k=0 \\ k\,\mbox{\tiny even}}}^{\infty}\sum_{j=0}^{\infty}\frac{(-1)^j}{(1+2k)(1+2j)^{2k+1}}&-&
\sum_{\substack{k=0 \\ k\,\mbox{\tiny odd}}}^{\infty}\sum_{j=0}^{\infty}\frac{(-1)^j}{(1+2k)(1+2j)^{2k+1}}\\
&=&\frac{1}{2}\tan^{-1}(\sinh(\pi/2)).
\end{eqnarray*}

\paragraph{Acknowledgements.}
EGD was partially funded by grant PID2021-127989OB-I00 (Ministerio de
Econom\'ia y Competitividad, Spain) and by grand TUR-RETOS2022-075.

\vfill\eject

\end{document}